\documentclass[reqno,11pt]{elsarticle}
\usepackage{amsfonts}
\usepackage{amssymb}
\usepackage{graphicx}
\usepackage{pstricks}
\usepackage{amsmath}
\usepackage{amsmath}
\usepackage{amsxtra}

\def\ker{\operatorname{ker}}
\def\card{\operatorname{card}}
\def\supp{\operatorname{supp}}
\def\rank{\operatorname{rank}}
\def\dim{\operatorname{dim}}
\def\Ran{\operatorname{Ran}}
\def\shift{\operatorname{shift}}

\newtheorem{theorem}{Theorem}[section]
\newtheorem{lemma}[theorem]{Lemma}
\newproof{proof}[theorem]{Proof}
\newtheorem{proposition}[theorem]{Proposition}
\newtheorem{corollary}[theorem]{Corollary}

\newtheorem{problem}[theorem]{Problem}

\newtheorem{example}[theorem]{Example}

\newtheorem{remark}[theorem]{Remark}

\numberwithin{equation}{section}
\begin{document}

\begin{frontmatter}

\title{ Quasinormality of powers of \\ commuting pairs of bounded operators}
\author{Ra\'{u}l E. Curto\footnote{The first named author was partially supported by NSF Grant
DMS-1302666.}}
\address{Department of Mathematics, The University of Iowa, Iowa City, Iowa
52242}
\ead{raul-curto@uiowa.edu}
\ead[url]{http://www.math.uiowa.edu/\symbol{126}rcurto/}

\author{Sang Hoon Lee\footnote{The second author of this paper was partially supported by NRF
(Korea) grant No. 2016R1D1A1B03933776.}}
\address{Department of Mathematics, Chungnam National University, Daejeon,
34134, Republic of Korea}
\ead{slee@cnu.ac.kr}
\author{Jasang Yoon\footnote{The third named author was partially supported by a grant from the
University of Texas System and the Consejo Nacional de Ciencia y Tecnolog%
\'{\i}a de M\'{e}xico (CONACYT).}}
\address{School of Mathematical and Statistical Sciences, The University of
Texas Rio Grande Valley, Edinburg, Texas 78539, USA}
\ead{jasang.yoon@utrgv.edu}

\begin{abstract}
We study jointly quasinormal and spherically quasinormal
pairs of commuting operators on Hilbert space, as well as their powers. \ We first prove that, up to a constant multiple, the only jointly quasinormal $2$-variable weighted shift is the Helton-Howe shift. \ Second, we show that a left invertible subnormal operator $T$ whose square $T^{2}$ is quasinormal must be quasinormal. \ Third, we generalize a characterization of quasinormality for subnormal operators in terms of their normal extensions to the case of commuting subnormal $n$-tuples. \ Fourth, we show that if a $2$-variable weighted shift $W_{\left(\alpha ,\beta \right) }$ and its powers $W_{\left(\alpha ,\beta \right)}^{(2,1)}$ and $W_{\left(\alpha ,\beta \right)}^{(1,2)}$ are all spherically quasinormal, then $W_{\left( \alpha ,\beta \right)}$ may not necessarily be jointly quasinormal. \ Moreover, it is possible for both $W_{\left(\alpha ,\beta \right)}^{(2,1)}$ and $W_{\left(\alpha ,\beta \right)}^{(1,2)}$ to be spherically quasinormal without $W_{\left(\alpha ,\beta \right)}$ being spherically quasinormal. \ Finally, we prove that, for $2$-variable weighted shifts, the common fixed points of the toral and spherical Aluthge transforms are jointly quasinormal.  
\end{abstract}


\end{frontmatter}


\section{\label{Int}Introduction}

Let $\mathcal{H}$ be a complex Hilbert space and let $\mathcal{B}(\mathcal{H})$ denote the algebra of bounded linear operators on $\mathcal{H}$. \ An operator $T\in \mathcal{B}(\mathcal{H})$ is said to be \textit{normal} if $T^{\ast}T=TT^{\ast }$, \textit{quasinormal} if $T$ commutes with $T^{\ast }T$,
i.e., $TT^{\ast }T=T^{\ast }T^{2}$, \textit{subnormal} if $T=N|_{\mathcal{H}%
} $, where $N$ is normal and $N(\mathcal{H})\subseteq \mathcal{H}$, and
\textit{hyponormal} if $T^{\ast }T\geq TT^{\ast }$.\ It is well known that 
\begin{equation*}
\begin{tabular}{l}
normal $\Longrightarrow $ quasinormal $\Longrightarrow $ subnormal $%
\Longrightarrow $ hyponormal.
\end{tabular}
\end{equation*}
For $S,T\in \mathcal{B}(\mathcal{H})$ let $[S,T]:=ST-TS$. \ We say that an
$n$-tuple $\mathbf{T}=(T_{1},\cdots ,T_{n})$ of operators on $\mathcal{H}$
is (jointly) \textit{hyponormal} if the operator matrix
\begin{equation*}
\lbrack \mathbf{T}^{\ast },\mathbf{T]:=}\left(
\begin{array}{llll}
\lbrack T_{1}^{\ast },T_{1}] & [T_{2}^{\ast },T_{1}] & \cdots & [T_{n}^{\ast
},T_{1}] \\
\lbrack T_{1}^{\ast },T_{2}] & [T_{2}^{\ast },T_{2}] & \cdots & [T_{n}^{\ast
},T_{2}] \\
\text{ \thinspace \thinspace \quad }\vdots & \text{ \thinspace \thinspace
\quad }\vdots & \ddots & \text{ \thinspace \thinspace \quad }\vdots \\
\lbrack T_{1}^{\ast },T_{n}] & [T_{2}^{\ast },T_{n}] & \cdots & [T_{n}^{\ast
},T_{n}]%
\end{array}%
\right)
\end{equation*}%
is positive on the direct sum of $n$ copies of $\mathcal{H}$ (cf. \cite{Ath}%
, \cite{bridge}, \cite{CMX}). \ The $n$-tuple $\mathbf{T}$ is said to be
\textit{normal} if $\mathbf{T}$ is commuting and each $T_{i}$ is normal, and
\textit{subnormal }if $\mathbf{T}$ is the restriction of a normal $n$-tuple
to a common invariant subspace. \ For $i,j,k\in \left\{ 1,2,\ldots
,n\right\} $, $\mathbf{T}$ is called \textit{matricially quasinormal} if
each $T_{i}$ commutes with each $T_{j}^{\ast }T_{k}$, $\mathbf{T}$ is
(jointly) \textit{quasinormal} if each $T_{i}$ commutes with each $%
T_{j}^{\ast }T_{j}$, and \textit{spherically quasinormal} if each $T_{i}$
commutes with $\sum_{j=1}^{n}T_{j}^{\ast }T_{j}$. \ As shown in \cite{AtPo} and \cite{Gle}, we have
\begin{equation}
\begin{tabular}{l}
normal $\Longrightarrow $ matricially quasinormal $\Longrightarrow $
(jointly) quasinormal \\
$\Longrightarrow $ spherically quasinormal $\Longrightarrow $ subnormal.%
\end{tabular}
\label{implication}
\end{equation}%
On the other hand, the results in \cite{CuYo6} and \cite{Gle} show that
the inverse implications in (\ref{implication}) do not hold.

It is well known that the only quasinormal $1$-variable weighted
shift is, up to a constant multiple, the (unweighted) unilateral shift $U_{+}:=\shift(1,1,\cdots )$. \ One of the aims of this paper is to show that there is a clear distinction between quasinormality in the single
operator case and spherical quasinormality for commuting pairs of operators. \ To describe our
results, we first need to introduce some notation and terminology. \ 
First, we consider the polar decomposition and Aluthge transforms for commuting
pairs $\mathbf{T}=(T_{1},T_{2})$. \ The reader will notice at once that results for pairs can be 
readily generalized to the case of commuting $n$-tuples of operators. \ 

For $T\in \mathcal{B}(\mathcal{H})$, the \textit{%
canonical polar decomposition} of $T$ is $T \equiv V|T|$ (with $\ker V = \ker T$) and the \textit{Aluthge transform} $%
\widetilde{T}$ of $T$ is $\widetilde{T}:=|T|^{\frac{1}{2}}U|T|^{\frac{1}{2}}$%
. \ The Aluthge transform was first introduced in \cite{Alu} and it has
attracted considerable attention over the last two decades (see, for instance, \cite{Ando},
\cite{CJL}, \cite{DySc}, \cite{JKP}, \cite{JKP2}, \cite{LLY} and \cite{Yam}). \ Roughly speaking, the idea behind the Aluthge transform is to
convert an operator into another operator which shares with the first one
many spectral properties, but which is closer to being a normal operator. \
For $T_{1},T_{2}\in \mathcal{B}(\mathcal{H})$, consider the pair $T=\left(
\begin{array}{c}
T_{1} \\
T_{2}%
\end{array}%
\right) $ as an operator from $\mathcal{H}$ into $\mathcal{H}\bigoplus
\mathcal{H}$, that is,
\begin{equation}
T=\left(
\begin{array}{c}
T_{1} \\
T_{2}%
\end{array}%
\right) :\mathcal{H}\rightarrow
\begin{tabular}{l}
$\mathcal{H}$ \\
$\bigoplus $ \\
$\mathcal{H}$%
\end{tabular}%
.  \label{setting 1}
\end{equation}%
Then, we have $T^{\ast }T=(T_{1}^{\ast },T_{2}^{\ast
})\left(
\begin{array}{c}
T_{1} \\
T_{2}%
\end{array}%
\right) =T_{1}^{\ast }T_{1}+T_{2}^{\ast }T_{2}$, so that we can define a
polar decomposition of $T\equiv \left(
\begin{array}{c}
T_{1} \\
T_{2}%
\end{array}%
\right) $ as follows:
\begin{equation}
T=\left(
\begin{array}{c}
T_{1} \\
T_{2}%
\end{array}%
\right) =\left(
\begin{array}{c}
V_{1} \\
V_{2}%
\end{array}%
\right) P=\left(
\begin{array}{c}
V_{1}P \\
V_{2}P%
\end{array}%
\right) \text{,}  \label{setting 2}
\end{equation}%
where $P=\sqrt{T_{1}^{\ast }T_{1}+T_{2}^{\ast }T_{2}}$. \ We then have $%
R:=(V_{1}^{\ast },V_{2}^{\ast })\left(
\begin{array}{c}
V_{1} \\
V_{2}%
\end{array}%
\right) =I$ on%
\begin{eqnarray*}
\left( \ker T\right) ^{\perp } &=&\left( \ker T_{1}\cap \ker
T_{2}\right) ^{\perp }=\left( \ker P\right) ^{\perp } \\
&=&\left( \ker \left(
\begin{array}{c}
V_{1} \\
V_{2}%
\end{array}%
\right) \right) ^{\perp }=\left( \ker V_{1}\cap \ker V_{2}\right) ^{\perp } .
\end{eqnarray*}%

For $\mathbf{T}\equiv (T_{1},T_{2})$, it is now natural to define the \textit{spherical }Aluthge transform of $\mathbf{T}$ as 

\begin{equation}
\widehat{\mathbf{T}}\equiv \widehat{\mathbf{(}T_{1},T_{2})}:=\left( \sqrt{P}V_{1}\sqrt{P},\sqrt{P}V_{2}\sqrt{P}\right) \text{ (cf. \cite{CuYo7}, \cite{CuYo6}, \cite{KiYo1}).}
\label{Def-Alu1}
\end{equation}
For a commuting pair of operators $\mathbf{T}=(T_{1},T_{2})$, we can also define
the \textit{toral} Aluthge transform by taking Aluthge transforms coordinate-wise:
\begin{equation*}
\widetilde{\mathbf{T}}\equiv \widetilde{\mathbf{(}T_{1},T_{2})}:=(\widetilde{%
T}_{1},\widetilde{T}_{2})\equiv (|T_{1}|^{\frac{1}{2}}U_{1}|T_{1}|^{\frac{1}{%
2}},|T_{2}|^{\frac{1}{2}}U_{2}|T_{2}|^{\frac{1}{2}})\text{ (cf. \cite{CuYo7}%
, \cite{CuYo6}, \cite{KiYo1}).}
\end{equation*}

Next, we recall the class of unilateral weighted shifts. \ For $\alpha \equiv \{\alpha
_{n}\}_{n=0}^{\infty }$ a bounded sequence of positive real numbers (called
\textit{weights}), let 
$$
W_{\alpha }\equiv \shift(\alpha _{0},\alpha_{1},\cdots ):\ell ^{2}(\mathbb{Z}_{+})\rightarrow \ell ^{2}(\mathbb{Z}_{+})
$$
be the associated unilateral weighted shift, defined by $W_{\alpha
}e_{n}:=\alpha _{n}e_{n+1}\;$(all $n\geq 0$), where $\{e_{n}\}_{n=0}^{\infty
}$ is the canonical orthonormal basis in $\ell ^{2}(\mathbb{Z}_{+})$. \ The
moments of $\alpha \equiv \{\alpha _{n}\}_{n=0}^{\infty }$ are given as
\begin{equation*}
\gamma _{k}\equiv \gamma _{k}(W_{\alpha }):=\left\{
\begin{tabular}{ll}
$1$, & $\text{if }k=0$ \\
$\alpha _{0}^{2}\cdots \alpha _{k-1}^{2}$, & $\text{if }k>0.$%
\end{tabular}%
\right.
\end{equation*}%
Similarly, consider double-indexed positive bounded sequences $\alpha _{%
\mathbf{k}},\beta _{\mathbf{k}}\in \ell ^{\infty }(\mathbb{Z}_{+}^{2})$, $%
\mathbf{k}\equiv (k_{1},k_{2})\in \mathbb{Z}_{+}^{2}$ and let $\ell ^{2}(%
\mathbb{Z}_{+}^{2})$\ be the Hilbert space of square-summable complex
sequences indexed by $\mathbb{Z}_{+}^{2}$. \ (Recall that $\ell ^{2}(\mathbb{%
Z}_{+}^{2})$ is canonically isometrically isomorphic to $\ell ^{2}(\mathbb{Z}%
_{+})\bigotimes \ell ^{2}(\mathbb{Z}_{+})$.) \ We define the $2$%
-variable weighted shift $W_{(\alpha ,\beta )}\equiv (T_{1},T_{2})$\ by%
\begin{equation*}
T_{1}e_{\mathbf{k}}:=\alpha _{\mathbf{k}}e_{\mathbf{k+}\varepsilon _{1}}
\end{equation*}%
\begin{equation*}
T_{2}e_{\mathbf{k}}:=\beta _{\mathbf{k}}e_{\mathbf{k+}\varepsilon _{2}},
\end{equation*}%
where $\mathbf{\varepsilon }_{1}:=(1,0)$ and $\mathbf{\varepsilon }%
_{2}:=(0,1)$. \ Clearly,
\begin{equation}
T_{1}T_{2}=T_{2}T_{1}\Longleftrightarrow \beta _{\mathbf{k+}\varepsilon
_{1}}\alpha _{\mathbf{k}}=\alpha _{\mathbf{k+}\varepsilon _{2}}\beta _{%
\mathbf{k}}\;\left( \text{all }\mathbf{k}\in \mathbb{Z}_{+}^{2}\right) .
\label{commuting}
\end{equation}%
From now on, we will consider only commuting $2$-variable weighted shifts. \ For basic properties for $2$-variable weighted shift $%
W_{(\alpha ,\beta )}$, we refer to \cite{CLY1} and \cite{CuYo1}.

Given $\mathbf{k}\equiv (k_{1},k_{2})\in \mathbb{Z}_{+}^{2}$, the moments of
$(\alpha ,\beta )$ of order $\mathbf{k}$ are
\begin{equation}
\gamma _{\mathbf{k}}\equiv \gamma _{\mathbf{k}}(W_{(\alpha ,\beta )}):=%
\begin{cases}
1, & \text{if }k_{1}=0\text{ and }k_{2}=0 \\
\alpha _{(0,0)}^{2}\cdots \alpha _{(k_{1}-1,0)}^{2}, & \text{if }k_{1}\geq 1%
\text{ and }k_{2}=0 \\
\beta _{(0,0)}^{2}\cdots \beta _{(0,k_{2}-1)}^{2}, & \text{if }k_{1}=0\text{
and }k_{2}\geq 1 \\
\alpha _{(0,0)}^{2}\cdots \alpha _{(k_{1}-1,0)}^{2}\beta
_{(k_{1},0)}^{2}\cdots \beta _{(k_{1},k_{2}-1)}^{2}, & \text{if }k_{1}\geq 1%
\text{ and }k_{2}\geq 1.%
\end{cases}
\label{moment0}
\end{equation}%
We remark that, due to the commutativity condition (\ref{commuting}), $%
\gamma _{\mathbf{k}}$ can be computed using any nondecreasing path from $%
(0,0)$ to $(k_{1},k_{2})$. \ We now recall a well known characterization of
subnormality for multivariable weighted shifts \cite{JeLu}, due to C. Berger
(cf. \cite[III.8.16]{Con}) and independently established by R. Gellar and
L.J. Wallen \cite{GeWa} in the single variable case: $\ W_{(\alpha ,\beta )}$
admits a commuting normal extension if and only if there is a probability
measure $\mu $ (which we call the \textit{Berger measure} of $W_{(\alpha
,\beta )}$) defined on the $2$-dimensional rectangle $R=[0,a_{1}]\times
\lbrack 0,a_{2}]$ $\left( \text{where }a_{i}:=\left\Vert T_{i}\right\Vert
^{2}\right) $ such that%
\begin{equation}
\gamma _{\mathbf{k}}(W_{(\alpha ,\beta )})=\int_{R}s^{k_{1}}t^{k_{2}}d\mu
(s,t)\text{, for all }\mathbf{k}\in \mathbb{Z}_{+}^{2}\text{ (called \textit{%
Berger's theorem}).}  \label{Berger Theorem}
\end{equation}%
For $0<a<1$ we let $S_{a}:=\shift(a,1,1,\cdots )$. \
Observe that $U_{+}$ and $S_{a}$ are subnormal, with Berger measures $\delta
_{1}$ and $(1-a^{2})\delta _{0}+a^{2}\delta _{1}$, respectively, where $%
\delta _{p}$ denotes the point-mass probability measure with support in the
singleton set $\{p\}$. \ We denote the class of subnormal pairs by $%
\mathfrak{H}_{\infty }$ and the class of commuting pairs of subnormal
operators by $\mathfrak{H}_{0}$; clearly, $\mathfrak{H}_{\infty } \subseteq \mathfrak{H}_{0}$.

Motivated in part by the results \cite{CLY1} and \cite{CLY2}, we now study some properties of matricially, jointly, and
spherically quasinormal pairs of commuting operators, including the following questions.

\begin{problem}
\label{Problem 1}Let $T$ be a subnormal operator, and assume that $T^{2}$ is quasinormal. \ Does it follow that $T$ is
quasinormal?
\end{problem}

\begin{problem}
\label{Problem 2}(i) \ Let $\mathbf{T}\equiv (T_{1},T_{2})\in \mathfrak{H}%
_{0} $ be spherically quasinormal, and let $m$ and $n$ be two positive integers. \ Is $\mathbf{T}^{\left( m,n\right)
}\equiv (T_{1}^{m},T_{2}^{n})$ spherically quasinormal?\newline
(ii) \ If $\mathbf{T}$ and $\mathbf{T}^{\left( 1,2\right) }\equiv
(T_{1},T_{2}^{2})$ are both spherically quasinormal, is $\mathbf{T}$
(jointly) quasinormal?\newline
(iii) \ If $\mathbf{T}^{\left( 2,1\right) }\equiv (T_{1}^{2},T_{2})$ and $\mathbf{T}^{\left( 1,2\right) }\equiv (T_{1},T_{2}^{2})$ are both spherically quasinormal, is $\mathbf{T}$ spherically quasinormal?
\end{problem}

\begin{problem}
\label{Problem 3}Let $\mathbf{T} \in \mathfrak{H}_{0}$, and assume that $\widehat{\mathbf{T}}=\mathbf{T}=\widetilde{\mathbf{T}}$. \ Is $\mathbf{T}$ (jointly) quasinormal?
\end{problem}

A subnormal operator is said to be {\it pure} if it has no nonzero normal orthogonal summands; that is, if there exists no nonzero subspace $\mathcal{M}$ of $\mathcal{H}$ invariant under $T$ such
that $T|_{\mathcal{M}}$ is normal, where $T|_{\mathcal{M}}$ denotes the
restriction of $T$ to $\mathcal{M}$. \ Since quasinormal operators are always subnormal, it makes sense to speak of pure quasinormal operators. \ In 1953, A. Brown obtained a characterization of pure quasinormal operators in terms of a tensor product of a unilateral shift $U$ (of finite or infinite multiplicity) and a positive operator $A$. \ Concretely, his result states that every pure quasinormal operator $T\in \mathcal{B}(\mathcal{H})$ is unitarily equivalent to $U_+\otimes A$ acting on $\ell^2(\mathbb{Z}_+)\otimes \mathcal{R}$, where $A\geq 0$ with $\ker A=\left\{0\right\}$; of course, $U = U_+ \otimes I_{\mathcal{R}}$, so that the multiplicity of $U$ is $\dim \mathcal{R}$ (\cite{Bro}, \cite[Theorem 3.2]{Con}; cf. Lemma \ref{Lemma 5}). \ Hence, it is natural to ask:

\begin{problem}
\label{Problem 4}Let $\mathbf{T}\equiv (T_{1},T_{2})$ be a
commuting pair of operators. \ Assume that $\mathbf{T}$ is spherically quasinormal and that $\mathbf{T}$ is pure (i.e., no nonzero normal orthogonal summands). \ Do there exist a (joint) isometry $\mathbf{U}%
=\left( U_{1},U_{2}\right) $ and a positive operator $P\geq 0$ such that $\mathbf{T}$ is unitarily equivalent to $\mathbf{U}\otimes P$?
\end{problem}

In Theorem \ref{spherically quasinormal criterion} we will obtain a characterization of quasinormality for 
subnormal $n$-tuples, which we believe is the right multivariable analogue of the key step in Brown's proof in the $1$-variable case.

We conclude this section by recording a result proved in \cite{CuYo8}. 

\begin{proposition}
\label{Prop1}(\cite[Theorem 2.2]{CuYo8}) \ For $\mathbf{T}\equiv (T_{1},T_{2})\in \mathfrak{H}_{0}$, the
following are equivalent:\newline
(i) $\mathbf{T}$ is spherically quasinormal;\newline
(ii) $\widehat{\mathbf{T}}=\mathbf{T}$.
\end{proposition}

\setlength{\unitlength}{1mm} \psset{unit=1mm}
\begin{figure}[th]
\begin{center}
\begin{picture}(135,70)

\psline{->}(20,20)(70,20)
\psline(20,35)(68,35)
\psline(20,50)(68,50)
\psline(20,65)(68,65)
\psline{->}(20,20)(20,70)
\psline(35,20)(35,68)
\psline(50,20)(50,68)
\psline(65,20)(65,68)

\put(12,16){\footnotesize{$(0,0)$}}
\put(31.5,16){\footnotesize{$(1,0)$}}
\put(46.5,16){\footnotesize{$(2,0)$}}
\put(61.5,16){\footnotesize{$(3,0)$}}

\put(25,21){\footnotesize{$\alpha_{(0,0)}$}}
\put(40,21){\footnotesize{$\alpha_{(1,0)}$}}
\put(55,21){\footnotesize{$\alpha_{(2,0)}$}}
\put(66,21){\footnotesize{$\cdots$}}

\put(25,36){\footnotesize{$\alpha_{(0,1)}$}}
\put(40,36){\footnotesize{$\alpha_{(1,1)}$}}
\put(55,36){\footnotesize{$\alpha_{(2,1)}$}}
\put(66,36){\footnotesize{$\cdots$}}

\put(25,51){\footnotesize{$\alpha_{(0,2)}$}}
\put(40,51){\footnotesize{$\alpha_{(1,2)}$}}
\put(55,51){\footnotesize{$\alpha_{(2,2)}$}}
\put(66,51){\footnotesize{$\cdots$}}

\put(26,66){\footnotesize{$\cdots$}}
\put(41,66){\footnotesize{$\cdots$}}
\put(56,66){\footnotesize{$\cdots$}}
\put(66,66){\footnotesize{$\cdots$}}

\psline{->}(35,14)(50,14)
\put(42,10){$\rm{T}_1$}
\psline{->}(10,35)(10,50)
\put(4,42){$\rm{T}_2$}

\put(11,34){\footnotesize{$(0,1)$}}
\put(11,49){\footnotesize{$(0,2)$}}
\put(11,64){\footnotesize{$(0,3)$}}

\put(20,26){\footnotesize{$\beta_{(0,0)}$}}
\put(20,41){\footnotesize{$\beta_{(0,1)}$}}
\put(20,56){\footnotesize{$\beta_{(0,2)}$}}
\put(21,66){\footnotesize{$\vdots$}}

\put(35,26){\footnotesize{$\beta_{(1,0)}$}}
\put(35,41){\footnotesize{$\beta_{(1,1)}$}}
\put(35,56){\footnotesize{$\beta_{(1,2)}$}}
\put(36,66){\footnotesize{$\vdots$}}

\put(50,26){\footnotesize{$\beta_{(2,0)}$}}
\put(50,41){\footnotesize{$\beta_{(2,1)}$}}
\put(50,56){\footnotesize{$\beta_{(2,2)}$}}
\put(51,66){\footnotesize{$\vdots$}}

\put(10,6){(i)}


\put(85,8){(ii)}

\psline{->}(95,14)(110,14)
\put(102,10){$\rm{T}_1$}
\psline{->}(77,35)(77,50)
\put(72,42){$\rm{T}_2$}

\psline{->}(80,20)(130,20)
\psline(80,35)(128,35)
\psline(80,50)(128,50)
\psline(80,65)(128,65)

\psline{->}(80,20)(80,70)
\psline(95,20)(95,68)
\psline(110,20)(110,68)
\psline(125,20)(125,68)

\put(75,16){\footnotesize{$(0,0)$}}
\put(91,16){\footnotesize{$(1,0)$}}
\put(106,16){\footnotesize{$(2,0)$}}
\put(121,16){\footnotesize{$(3,0)$}}

\put(85,22){\footnotesize{$\sqrt{\frac{2}{3}}$}}
\put(100,22){\footnotesize{$\sqrt{\frac{5}{6}}$}}
\put(115,22){\footnotesize{$\sqrt{\frac{14}{15}}$}}
\put(126,21){\footnotesize{$\cdots$}}

\put(85,37){\footnotesize{$\sqrt{\frac{1}{3}}$}}
\put(100,37){\footnotesize{$\sqrt{\frac{1}{3}}$}}
\put(115,37){\footnotesize{$\sqrt{\frac{1}{3}}$}}
\put(126,36){\footnotesize{$\cdots$}}

\put(85,52){\footnotesize{$\sqrt{\frac{1}{3}}$}}
\put(100,52){\footnotesize{$\sqrt{\frac{1}{3}}$}}
\put(115,52){\footnotesize{$\sqrt{\frac{1}{3}}$}}
\put(126,51){\footnotesize{$\cdots$}}

\put(85,66){\footnotesize{$\cdots$}}
\put(100,66){\footnotesize{$\cdots$}}
\put(115,66){\footnotesize{$\cdots$}}
\put(126,66){\footnotesize{$\cdots$}}

\put(80,26){\footnotesize{$\sqrt{\frac{1}{3}}$}}
\put(80,41){\footnotesize{$\sqrt{\frac{2}{3}}$}}
\put(80,56){\footnotesize{$\sqrt{\frac{2}{3}}$}}
\put(81,66){\footnotesize{$\vdots$}}

\put(95,26){\footnotesize{$\sqrt{\frac{1}{6}}$}}
\put(95,41){\footnotesize{$\sqrt{\frac{2}{3}}$}}
\put(95,56){\footnotesize{$\sqrt{\frac{2}{3}}$}}
\put(96,66){\footnotesize{$\vdots$}}

\put(110,26){\footnotesize{$\sqrt{\frac{1}{15}}$}}
\put(110,41){\footnotesize{$\sqrt{\frac{2}{3}}$}}
\put(110,56){\footnotesize{$\sqrt{\frac{2}{3}}$}}
\put(111,66){\footnotesize{$\vdots$}}
\end{picture}
\end{center}
\caption{Weight diagram of a generic 2-variable weighted shift and weight
diagram of the 2-variable weighted shift in Example \protect\ref{ex1},
respectively.}
\label{Figure1}
\end{figure}
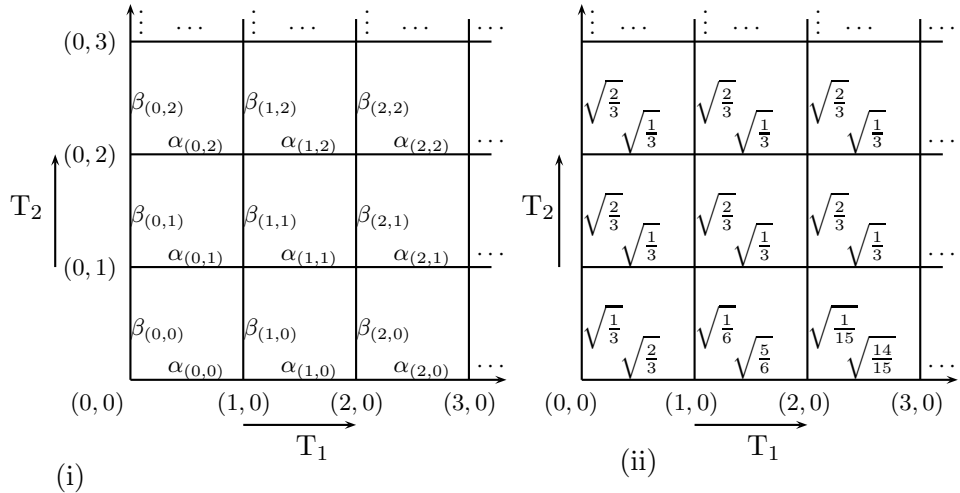




\section{\label{Quasinormal powers}Subnormal Pairs with Spherically Quasinormal Powers}

We first turn our attention to powers of quasinormal operators. \ For $T\in \mathcal{B}(\mathcal{H})$, it is well known that the
hyponormality of $T$ does not imply the hyponormality of $T^{2}$ \cite{Hal2}. \ However, for a unilateral weighted shift $W_{\alpha }$, the hyponormality
of $W_{\alpha }$ (detected by the condition $\alpha _{k}\leq \alpha _{k+1}$
for all $k\geq 0$) clearly implies the hyponormality of every power $%
W_{\alpha }^{m}\;(m\geq 1)$. \ It is also well known that the subnormality of $%
T$ implies the subnormality of $T^{m}$ $(m\geq 2)$. \ The converse
implication, however, is false. \ In fact, the subnormality of all powers $%
T^{m}\;(m\geq 2)$ does not necessarily imply the subnormality of $T$, even
if $T\equiv W_{\omega }$ is a unilateral weighted shift; e.g., for $0<a<b<1$ the square of $\shift (a,b,1,1,\cdots)$ is subnormal (\cite{Hal2}, \cite{JeLu}, \cite{Shi}, \cite{Sta2}). \ Lemma \ref{Lemma 5} says that if $T\in
\mathcal{B}(\mathcal{H})$ is (pure) quasinormal, then $T=U\left\vert
T\right\vert \cong U_{+}\otimes A$ acting on $\ell^2(\mathbb{Z}_+)
\otimes \mathcal{R}$, where $A\geq 0$ with $\mathrm{\ker }A=\left\{ 0\right\} $, $U\cong U_{+}\otimes
I_{\mathcal{R}}$, and $\left\vert T\right\vert \cong I_{\ell^2(\mathbb{Z}_+)}\otimes
A $. \ Hence, for $m\geq
1$
\begin{equation}
\begin{tabular}{l}
$T^{m}\cong U_{+}^{m}\otimes A^{m}\text{ and }\left( T^{\ast }\right)
^{m}\cong \left( U_{+}^{\ast }\right) ^{m}\otimes A^{m}\text{,}$%
\end{tabular}
\label{powers1}
\end{equation}%
so that
\begin{equation}
\begin{tabular}{l}
$\left( T^{\ast }\right) ^{m}T^{m}\cong I\otimes A^{2m}$.
\end{tabular}
\label{powers}
\end{equation}%
Therefore, the commutator $\left[ T^{m},\left( T^{\ast }\right) ^{m}T^{m}\right]$ is equal to $0$; that is, 
$T^{m}$ is quasinormal for all $m\geq 1$. \ In view of this, it is natural to ask: If $%
T^{2}$ is (pure) quasinormal, is $T$ (pure) quasinormal? \ Without a restriction on $T$, it is easy to answer this question in the negative. \ For, let $T$ be a nonzero nilpotent operator of
order two, that is, $T^{2}=0$. \ Then, $T^{2}$ is quasinormal, but $T$ is
not quasinormal. \ Of course, such an operator cannot be subnormal. \ Let us assume then that $T$ is subnormal.

The following problem is a refinement of Problem \ref{Problem 1}.

\begin{problem}
\label{P1}Let $T$ be subnormal and assume that $T^{2}$ is (pure) quasinormal. \ Is it true that $T$ is (pure) quasinormal?
\end{problem}
We now provide an almost complete answer to Problem \ref{P1}. \ First, we need an auxiliary lemma, of independent interest. \ It is similar in spirit to \cite[Lemma 3.1]{Con}. 

\begin{lemma} \label{keylemma}
Let $T$ be a subnormal operator on $\mathcal{H}$, with normal extension $N$ acting on $\mathcal{K} \supseteq \mathcal{H}$, and write 
\begin{equation*}
N=\left(
\begin{array}{cc}
T & A \\
0 & B^{\ast }%
\end{array}%
\right) \text{.}
\end{equation*}
Then $T$ is quasinormal if and only if $A^{\ast }T=0$.
\end{lemma}

\begin{proof}
We calculate
\begin{equation*}
NN^{\ast }N=\left(
\begin{array}{cc}
TT^{\ast }T+AA^{\ast }T & TT^{\ast }A+A\left( A^{\ast }A+BB^{\ast }\right)
\\
B^{\ast }A^{\ast }T & B^{\ast }\left( A^{\ast }A+BB^{\ast }\right)%
\end{array}%
\right)
\end{equation*}%
and
\begin{equation*}
N^{\ast }NN=\left(
\begin{array}{cc}
T^{\ast }TT & T^{\ast }TA+T^{\ast }AB^{\ast } \\
A^{\ast }TT & A^{\ast }TA+\left(A^{\ast}A+BB^{\ast }\right) B^{\ast}
\end{array}%
\right)
\end{equation*}%
Since $NN^{\ast }N=NN^{\ast }N$, we have
\begin{equation*}
TT^{\ast }T+AA^{\ast }T=T^{\ast }TT.
\end{equation*}%
Then the commutator of $T$ and $T^{\ast }T$ is $\left[ T,T^{\ast }T\right] =-AA^{\ast }T$. \ It follows that $T$ is quasinormal if and only if $AA^{\ast }T=0$ if and only if $\Ran T\subseteq
\ker AA^{\ast }=\ker A^{\ast }$ if and only if $A^{\ast
}T=0$. \ This completes the proof.
\end{proof}

\begin{theorem}
\label{Theorem 1}Let $T\in \mathcal{B}(\mathcal{H})$ be subnormal and assume that $T^{2}$
is quasinormal. \ If $T$ is bounded below (i.e., left invertible), then $T$ is quasinormal.
\end{theorem}

\begin{proof}
Since $T\in \mathcal{B}(\mathcal{H})$ is subnormal, we consider a normal
extension $N$ of $T$ such that
\begin{equation*}
N=\left(
\begin{array}{cc}
T & A \\
0 & B^{\ast }%
\end{array}%
\right) \text{.}
\end{equation*}
Assume that $T^{2}$ is quasinormal, and consider the matricial representation of $N^{2}$, that is,
\begin{equation*}
N^{2}=\left(
\begin{array}{cc}
T^{2} & TA+AB^{\ast } \\
0 & B^{\ast 2}%
\end{array}%
\right) \text{.}
\end{equation*}%
From Lemma \ref{keylemma}, we know that $T^{2}$ is quasinormal if and only if $\left(
TA+AB^{\ast }\right) ^{\ast }T^{2}=0$ if and only if $A^{\ast }T^{\ast
}T^{2}+BA^{\ast }T^{2}=0$. \ Recall that $N^{\ast }N=NN^{\ast }$, i.e.,%
\begin{equation*}
\left\{
\begin{array}{c}
T^{\ast }T=TT^{\ast }+AA^{\ast } \\
A^{\ast }T=B^{\ast }A^{\ast } \\
A^{\ast }A+BB^{\ast }=B^{\ast }B.%
\end{array}%
\right .
\end{equation*}%
We thus have
\begin{eqnarray*}
A^{\ast }T^{\ast }TT &=&A^{\ast }\left( TT^{\ast }+AA^{\ast }\right) T \\
&=&A^{\ast }TT^{\ast }T+A^{\ast }AA^{\ast }T
\end{eqnarray*}%
and
\begin{equation*}
BA^{\ast }TT=B\left( A^{\ast }T\right) T=BB^{\ast }A^{\ast }T.
\end{equation*}%
Since we are assuming that $T^{2}$ is quasinormal, the previous calculation reveals that 
\begin{eqnarray*}
0 &=&A^{\ast }T^{\ast }T^{2}+BA^{\ast }T^{2} \\
&=&A^{\ast }TT^{\ast }T+A^{\ast }AA^{\ast }T+BB^{\ast }A^{\ast }T \\
&=&A^{\ast }TT^{\ast }T+\left( A^{\ast }A+BB^{\ast }\right) A^{\ast }T \\
&=&\left( A^{\ast }T\right) T^{\ast }T+B^{\ast }B\left( A^{\ast }T\right) \\
&=&\left( R_{T^{\ast }T}+L_{B^{\ast }B}\right) \left( A^{\ast }T\right) ,
\end{eqnarray*}%
where $R_{X}$ and $L_{X}$ denote the right and left multiplication operators by $X$ acting on $%
\mathcal{B}(\mathcal{H})$. \ It follows that $A^{\ast }T\in \ker \left( R_{T^{\ast }T}+L_{B^{\ast }B}\right) $. \ However, by the spectral
mapping theorem for the left spectrum $\sigma _{\ell } $ \cite{Appl}, we
have
\begin{equation*}
\sigma _{\ell }\left( R_{T^{\ast }T}+L_{B^{\ast }B}\right) =\left\{ \lambda
+\mu :\left( \lambda ,\mu \right) \in \sigma _{\ell }\left( R_{T^{\ast
}T},L_{B^{\ast }B}\right) \right\} \text{.}
\end{equation*}
If $0\in \sigma _{\ell }\left( R_{T^{\ast }T}+L_{B^{\ast }B}\right) $, then $%
0=\lambda +\mu $ for some $\left( \lambda ,\mu \right) \in \sigma _{\ell
}\left( R_{T^{\ast }T},L_{B^{\ast }B}\right) $. \ But
\begin{equation*}
\sigma _{\ell }\left( R_{T^{\ast }T},L_{B^{\ast }B}\right) =\sigma
_{r}\left( T^{\ast }T\right) \times \sigma _{\ell }\left( B^{\ast }B\right)
\subseteq \mathbb{R}_{+}\times \mathbb{R}_{+} \; \textrm{(\cite{Cu3}, \cite{CuFi}).} 
\end{equation*}%
Thus, if $0\in \sigma _{\ell }\left( R_{T^{\ast }T}+L_{B^{\ast }B}\right) $,
then $0=\lambda +\mu $, with $\lambda,\mu \in \mathbb{R}_+$. \ It follows that $\lambda =\mu =0$. \ This
implies that $0\in \sigma _{r}\left( T^{\ast }T\right)$ and $0\in \sigma
_{\ell }\left( B^{\ast }B\right) $, and therefore $0\in \sigma \left( T^{\ast
}T\right)$ and $0\in \sigma \left( B^{\ast }B\right) $. \ Therefore, we
have $0\in \sigma _{\ell }\left( T\right) \cap \sigma _{\ell }\left(
B\right) $, that is, neither $T$ nor $B$ is bounded below. \ It follows that, under the assumption that $T$ is bounded below, the operator $R_{T^*T}+L_{B*B}$ is injective, and then $A^*T=0$, as desired.
\end{proof}

\begin{corollary}
\label{Cor3}Let $T\in \mathcal{B}(\mathcal{H})$ be subnormal and $T^{2}$ be
pure quasinormal. \ If $T$ is bounded below, then $T$ is pure
quasinormal.
\end{corollary}

\begin{proof}
Quasinormality is clear from Theorem \ref{Theorem 1}. \ If $T$ is not pure,
then there is a nonzero reducing subspace $\mathcal{M}$ of $\mathcal{H}$ such such
that $T|_{\mathcal{M}}$ is normal, where $T|_{\mathcal{M}}$ denotes the
restriction of $T$ to $\mathcal{M}$. \ Since $T^{2}|_{\mathcal{M}}$
is also normal, $T^{2}$ is not pure, which is a contradiction. \ Therefore, $T$ is pure.
\end{proof}

\begin{remark}
As described in Lemma \ref{Lemma 5}, a pure quasinormal operator $T$ is, up to unitary equivalence, of the form
$T=U_{+} \otimes A$, where $A \ge 0$ acts on $\mathcal{R}$. \ Then $T^2= U_{+}^2\otimes A^2$. \ It follows from Corollary \ref{Cor3} that a bounded below subnormal operator $T$ whose square is pure quasinormal must be of the form $U_{+} \otimes A$, where $\dim \mathcal{R}$ is either {\it infinite} or {\it finite and even}.
\end{remark}

The multivariable analogues of Theorem \ref
{Theorem 1} are highly nontrivial. \ Thus, at present Problem \ref{Problem 4} is rather challenging.

We now present a multivariable analogue of a key step in Brown's characterization of quasinormality within the class of subnormal operators. \ We believe Theorem \ref{spherically quasinormal criterion} below can be used to solve Problem \ref{Problem 4}. \ Although we present our results for commuting pairs of operators, the reader will easily see that the same statements work well for commuting $n$-tuples of operators, when $n>2$. 

Let $\mathbf{T} \equiv (T_1,T_2)$ be a (commuting) subnormal pair of operators on a Hilbert space $\mathcal{H})$, and let $\mathbf{N} \equiv (N_1,N_2)$ be a normal extension of $\mathbf{T}$ acting on $\mathcal{K} \supseteq \mathcal{H}$. \ For $i=1,2$, write
\begin{equation} 
N_i=\left(
\begin{array}{cc}
T_i & A_i \\
0 & B_i^{ \ast } \label{normal extension}
\end{array}%
\right) \text{.}
\end{equation}
Then
\begin{equation} \label{N1eq}
N_1^*N_1+N_2^*N_2=\left(
\begin{array}{cc} 
T_1^*T_1+T_2^*T_2 & T_1^*A_1+T_2^*A_2 \\
A_1^*T_1+A_2^*T_2 & A_1^*A_1+B_1B_1^*+A_2^*A_2+B_2B_2^*
\end{array}%
\right) \text{.}
\end{equation}%
As usual, let $P:=\sqrt{T_1^*T_1+T_2^*T_2}$.

\begin{remark} \label{rem1}
From Lemma \ref{keylemma}, we know that if $T_1$ and $T_2$ are quasinormal then $A_1^*T_1=A_2^*T_2=0$. \ For $i=1,2$, it follows that
\begin{eqnarray}
0&=&N_i(N_1^*N_1+N_2^*N_2)-(N_1^*N_1+N_2^*N_2)N_i \\
&=&
\left(
\begin{array}{cc}
T_i & A_i \\
0 & B_i^{\ast }
\end{array}
\right)
\left(
\begin{array}{cc}
P^2 & 0 \\
0 & *
\end{array}
\right)-
\left(
\begin{array}{cc}
P^2 & 0 \\
0 & *
\end{array}
\right)
\left(
\begin{array}{cc}
T_i & A_i \\
0 & B_i^{\ast }
\end{array}
\right)
\\
&=&
\left(
\begin{array}{cc}
T_iP^2-P^2T_i & * \\
0 & *
\end{array}
\right) .
\end{eqnarray}
It follows that $(T_1,T_2)$ is spherically quasinormal. \ In short, if $\mathbf{T}$ is subnormal and each of $T_1$ and $T_2$ is quasinormal, then $\mathbf{T}$ is spherically quasinormal.
\end{remark}

\begin{corollary}
Let $\mathbf{T}$ be subnormal, and assume that $T_i$ is bounded below and $T_i^2$ is quasinormal $(i=1,2)$. \ Then $\mathbf{T}$ is spherically quasinormal.
\end{corollary}

\begin{proof}
Straightforward from Theorem \ref{Theorem 1} and Remark \ref{rem1}.
\end{proof}

\begin{theorem} \label{spherically quasinormal criterion}
Let $\mathbf{T}$ be subnormal, with normal extension $\mathbf{N}$. \ Then $\mathbf{T}$ is spherically quasinormal if and only if $A_1^*T_1+A_2^*T_2=0$.
\end{theorem}

\begin{proof}
Using (\ref{N1eq}) and the calculation in Remark \ref{rem1}, and looking at the $(1,1)$ entry of $N_i(N_1^*N_1+N_2^*N_2)-(N_1^*N_1+N_2^*N_2)N_i \; (i=1,2)$, it is clear that 
$$
[T_i,P^2]=-A_i(A_1^*T_1+A_2^*T_2).
$$
It follows that $(T_1,T_2)$ is spherically quasinormal if and only if $A_i(A_1^*T_1+A_2^*T_2)=0$ for $i=1,2$. \ Consider now the equation
\begin{eqnarray} \label{eq200}
\left(
\begin{array}{c}
A_1 \\
A_2
\end{array}
\right)
\left(
\begin{array}{cc}
A_1^* & A_2^*
\end{array}
\right)
\left(
\begin{array}{c}
T_1 \\
T_2
\end{array}
\right)
=
\left(
\begin{array}{c}
0 \\
0
\end{array}
\right) .
\end{eqnarray}
Since $\ker ZZ^*=\ker Z^*$ for all operators $Z$, we see that (\ref{eq200}) is equivalent to 
\begin{eqnarray*} 
\left(
\begin{array}{cc}
A_1^* & A_2^*
\end{array}
\right)
\left(
\begin{array}{c}
T_1 \\
T_2
\end{array}
\right)
=
\left(
\begin{array}{c}
0 \\
0
\end{array}
\right) ,
\end{eqnarray*}
and this is equivalent to $A_1^*T_1+A_2^*T_2=0$, as desired.
\end{proof}

We now consider the case of joint quasinormality for a subnormal pair, and how the normal extension detects it. \ Recall that a commuting pair $\mathbf{T}$ is (jointly) quasinormal if $T_i$ commutes with $T_j^*T_j$ for all $i,j=1,2$. \ Also, recall the Fuglede-Putnam Theorem \cite[12.5]{ConwayOTbook}: if $N$ and $M$ are normal operators and $T$ is an operator such that $NT=TM$, then $N^*T=TM^*$. 

\begin{theorem}
\label{Cor4}Let $\mathbf{T}\equiv (T_{1},T_{2})$ be a subnormal pair, with normal extension $\mathbf{N}$. \ Then $\mathbf{T}$ is (jointly) quasinormal if and only if $A_{i}^{\ast }T_{j}=0$ $\left( i,j=1,2\right) $.
\end{theorem}

\begin{proof}
As in (\ref{normal extension}), write 
\begin{equation*}
\left( N_{1},N_{2}\right) =\left( \left(
\begin{array}{cc}
T_{1} & A_{1} \\
0 & B_{1}^{\ast }%
\end{array}%
\right) ,\left(
\begin{array}{cc}
T_{2} & A_{2} \\
0 & B_{2}^{\ast }%
\end{array}%
\right) \right) \text{.}
\end{equation*}%
For $i,j\in \left\{ 1,2\right\} $, $N_{i}N_{j}=N_{j}N_{i}$ and $N_{i}$ is
normal. \ By the Fuglede-Putnam Theorem we have
\begin{equation}
N_{i}N_{j}^{\ast }=N_{j}^{\ast }N_{i}.  \label{com}
\end{equation}%
For $i,j,k\in \left\{ 1,2\right\} $, we calculate
\begin{equation*}
N_{i}N_{j}^{\ast }N_{k}=\left(
\begin{array}{cc}
T_{i}T_{j}^{\ast }T_{k}+A_{i}A_{j}^{\ast }T_{k} & T_{i}T_{j}^{\ast
}A_{k}+A_{i}\left( A_{j}^{\ast }A_{k}+B_{j}B_{k}^{\ast }\right)  \\
B_{i}^{\ast }A_{j}^{\ast }T_{k} & B_{i}^{\ast }\left( A_{j}^{\ast
}A_{k}+B_{j}B_{k}^{\ast }\right)
\end{array}%
\right)
\end{equation*}%
and
\begin{equation*}
N_{j}^{\ast }N_{k}N_{i}=\left(
\begin{array}{cc}
T_{j}^{\ast }T_{k}T_{i} & T_{j}^{\ast }T_{k}A_{i}+T_{j}^{\ast}A_{k}B_{i}^{\ast } \\
A_{j}^{\ast }T_{k}T_{i} & A_{j}^{\ast }T_{k}A_{i}+\left( A_{j}^{\ast}A_{k}+B_{j}B_{k}^{\ast }\right) B_{i}^{\ast }%
\end{array}%
\right)
\end{equation*}%
Since $N_{i}N_{j}^{\ast }N_{k}=N_{j}^{\ast }N_{k}N_{i} \; (i,j,k=1,2)$, we have
\begin{equation}
T_{i}T_{j}^{\ast }T_{k}+A_{i}A_{j}^{\ast }T_{k}=T_{j}^{\ast}T_{k}T_{i} \; (i,j,k=1,2). \label{equ1}
\end{equation}
Observe that for $i=j$ we have $\left[T_{i},T_{i}^{\ast }T_{k}\right] =-A_{i}A_{i}^{\ast }T_{k}$.

$\left( \Longrightarrow \right)$: \ If $\mathbf{T}\equiv (T_{1},T_{2})$ is
(jointly) quasinormal, then by (\ref{equ1}), $\left[ T_{i},T_{i}^{\ast }T_{k}\right] =-A_{i}A_{i}^{\ast }T_{k}=0 \; (i,k=1,2)$. \ It follows that $A_{i}^{\ast }T_{k}=0 \; (i,k=1,2)$, as desired.

$\left( \Longleftarrow \right) :$ \ This is clear from (\ref{equ1}).
\end{proof}

Recall now that a commuting pair $\mathbf{T}$ is matricially quasinormal if $T_i$ commutes with $T_j^*T_k$ for all $i,j,k=1,2$. 

\begin{corollary}
Let $\mathbf{T}\equiv (T_{1},T_{2})$ be a subnormal pair, with normal extension $\mathbf{N}$. \ Then $\mathbf{T}$ is matricially quasinormal if and only if $A_{i}A_{j}^{\ast }T_{k}=0$ $\left( i,j,k=1,2\right) $.
\end{corollary}

\begin{proof}
Straightforward from the proof of Theorem \ref{Cor4} and the definition of matricial quasinormality.
\end{proof}




\section{\label{JQ2VWS}Powers of Spherically Quasinormal $2$-variable Weighted Shifts}

In this section we focus on the class of $2$-variable weighted shifts. \ We mainly study spherical quasinormality, a notion that is emerging, in the multivariable context, as the most appropriate generalization of the classical notion of quasinormality for single operators. \ We start with two simple results, which help to validate the previous comment. \ We first observe that there is no matricially quasinormal $2$-variable
weighted shift. \ Next, we prove that the only (jointly) quasinormal $2$-variable weighted
shift is, up to a constant multiple, the Helton-Howe shift, that is, the shift of the form $\left( I\otimes U_{+}, U_{+}\otimes I\right)$, with Berger measure $\delta_1 \times \delta_1$. \ (We say that two commuting pairs $(S_1,S_2)$ and $(T_1,T_2)$ {\it differ by a constant multiple} if $S_i=r_iT_i$ for some $r_1,r_2>0$.) \ To this end, we need the following result which was announced in \cite{CuYo7} and established in \cite{CuYo6} (cf. \cite{CuYo8}).

\begin{theorem}
\label{Quasinormal} \ For a $2$-variable weighted shift $W_{\left( \alpha
,\beta \right) }=\left( T_{1},T_{2}\right) $, the following statements are
equivalent:\newline
(i) $\ W_{\left( \alpha ,\beta \right) }\equiv (T_{1},T_{2})\in \mathfrak{H}%
_{0}$ is a spherically quasinormal $2$-variable weighted shift; \newline
(ii) \ $T_{1}^{\ast }T_{1}+T_{2}^{\ast }T_{2}=C\cdot I$, for some positive constant $C$; \newline
(iii) \ for all $\mathbf{k}\equiv (k_{1},k_{2})\in \mathbb{Z}_{+}^{2}$, $%
\alpha _{(k_{1},k_{2})}^{2}+\beta _{(k_{1},k_{2})}^{2}=C$, for some positive constant $C>0$; \newline
(iv) \ for all $\mathbf{k}\equiv (k_{1},k_{2})\in \mathbb{Z}_{+}^{2}$, 
$\gamma_{\mathbf{k}+\varepsilon_1}+\gamma_{\mathbf{k}+\varepsilon_2}=C\gamma_{\mathbf{k}}$, for some positive constant $C>0$.
\end{theorem}

\begin{remark} A commuting pair $(T_1,T_2)$ is said to be a spherical isometry if $T_1^*T_1+T_2^*T_2=I$. \ It is well known that spherical isometries are subnormal (\cite{Ath1}, \cite{EsPu2}). \ Theorem \ref{Quasinormal} asserts that every spherically quasinormal $2$-variable weighted shift is, up to a constant multiple, a spherical isometry.
\end{remark}

Let us consider now matricial quasinormality for $2$-variable weighted shifts. \ Since $T_1$ must commute with $T_1^*T_2$, it follows easily that $T_1$ must be normal when restricted to $\Ran T_2$. \ In particular, the restriction of $T_1$ to each row with index $k_2$ bigger than zero will be normal. \ This contradicts the fact that there are no normal unilateral weighted shifts. \ Therefore, there are no matricially quasinormal $2$-variable weighted shifts.

Next, we examine (joint) quasinormality. \ Let us assume that $%
W_{\left( \alpha ,\beta \right) }\equiv (T_{1},T_{2})\in \mathfrak{H}_{0}$
is (jointly) quasinormal, i.e., $T_{i}$ commutes with $T_{j}^{\ast }T_{j}$ for $i,j\in \left\{ 1,2\right\} $. \ Thus, for all $(k_{1},k_{2})\in \mathbb{Z}_{+}^{2}$ we obtain that
\begin{equation}
\begin{tabular}{l}
$T_{1}T_{2}^{\ast }T_{2}\left( e_{(k_{1},k_{2})}\right) =T_{2}^{\ast
}T_{2}T_{1}\left( e_{(k_{1},k_{2})}\right) $ \\
$\Longleftrightarrow \beta _{(k_{1},k_{2})}^{2}T_{1}\left(
e_{(k_{1},k_{2})}\right) =\alpha _{(k_{1},k_{2})}\beta
_{(k_{1}+1,k_{2})}T_{2}^{\ast }\left( e_{(k_{1}+1,k_{2}+1)}\right) $ \\
$\Longleftrightarrow \beta _{(k_{1},k_{2})}^{2}\alpha
_{(k_{1},k_{2})}=\alpha _{(k_{1},k_{2})}\beta _{(k_{1}+1,k_{2})}^{2}$ \\
$\Longleftrightarrow \beta _{(k_{1},k_{2})}=\beta _{(k_{1}+1,k_{2})}$.%
\end{tabular}
\label{cond1}
\end{equation}%
and%
\begin{equation}
\begin{tabular}{l}
$T_{2}T_{1}^{\ast }T_{1}\left( e_{(k_{1},k_{2})}\right) =T_{1}^{\ast
}T_{1}T_{2}\left( e_{(k_{1},k_{2})}\right) $ \\
$\Longleftrightarrow \alpha _{(k_{1},k_{2})}=\alpha _{(k_{1},k_{2}+1)}$.%
\end{tabular}
\label{cond2}
\end{equation}%
Recall now that jointly quasinormal pairs are always spherically quasinormal, and apply Theorem \ref{Quasinormal} to (\ref{cond1}) and (\ref{cond2}). \ We see that
\begin{equation}
\alpha _{(k_{1},k_{2})}=\alpha _{(k_{1}+1,k_{2})}\text{ and }\beta
_{(k_{1},k_{2})}=\beta _{(k_{1},k_{2}+1)}\text{.}  \label{cond3}
\end{equation}
Since (joint) quasinormality implies spherical quasinormality, by (\ref{cond1}) and (\ref{cond2}) we have 
\begin{equation}
\alpha _{(k_{1},k_{2})}^{2}+\beta _{(k_{1},k_{2})}^{2}=\alpha
_{(k_{1}+1,k_{2})}^{2}+\beta _{(k_{1}+1,k_{2})}^{2}\Longrightarrow \alpha
_{(k_{1},k_{2})}=\alpha _{(k_{1}+1,k_{2})}  \label{cond4}
\end{equation}%
and
\begin{equation}
\alpha _{(k_{1},k_{2})}^{2}+\beta _{(k_{1},k_{2})}^{2}=\alpha
_{(k_{1},k_{2}+1)}^{2}+\beta _{(k_{1},k_{2}+1)}^{2}\Longrightarrow \beta
_{(k_{1},k_{2})}=\beta _{(k_{1},k_{2}+1)}\text{.}  \label{cond5}
\end{equation}
Let $W_{\alpha ^{(k_{2})}}$ (resp. $W_{\beta ^{(k_{1})}}$) denote the unilateral weighted shift of the associated with the $k_{2}$-th horizontal row (resp. $k_{1}$-th vertical
column) in the weighted diagram of $W_{\left( \alpha ,\beta \right)
}\equiv (T_{1},T_{2})$. \ Condition (\ref{cond4})
implies that
\begin{equation}
W_{\alpha ^{(k_{2})}}=\shift(\alpha _{\left( 0,k_{2}\right) },\alpha
_{\left( 1,k_{2}\right) },\cdots )=\alpha_{(0,k_2)} \cdot U_{+}.  \label{cond6}
\end{equation}%
Similarly, Condition (\ref{cond5}) implies that
\begin{equation}
W_{\beta ^{(k_{1})}}=\shift(\beta _{\left( k_{1},0\right) },\beta
_{\left( k_{1},10\right) },\cdots )=\beta_{(k_1,0)}\cdot U_{+}.  \label{cond7}
\end{equation}%
Thus, by (\ref{cond6}) and (\ref{cond7}), we obtain that, up to a constant multiple, the only (jointly) quasinormal $2$-variable weighted shift is $\left( I\otimes U_{+},U_{+}\otimes I\right) $.

We summarize the previous analysis in the following result.

\begin{theorem}
\label{Theorem 2}Up to a constant multiple, the only (jointly) quasinormal $2$-variable weighted shift is the Helton-Howe shift $\left( I\otimes U_{+},U_{+}\otimes I\right) $.
\end{theorem}

We now discuss four key questions in the study of spherical quasinormality. \ The following two Problems are suitable restatements of Problems \ref{Problem 2} and \ref{Problem 3}, respectively.

\begin{problem}
\label{P2}(i) \ Let $\mathbf{T}\equiv (T_{1},T_{2})\in \mathfrak{H}_{0}$ be
spherically quasinormal. \ Does it follow that $\mathbf{T}^{\left( m,n\right) }\equiv
(T_{1}^{m},T_{2}^{n})$ is spherically quasinormal?\newline
(ii) \ Let $\mathbf{T}$ be spherically quasinormal and assume that $\mathbf{T}^{\left( 1,2\right) }\equiv
(T_{1}^2,T_{2})$ is also spherically quasinormal. \ Does it follow that $\mathbf{T}$
is jointly quasinormal?\newline
(iii) \ Let $\mathbf{T}\equiv (T_{1},T_{2})\in \mathfrak{H}_{0}$, and assume that both $\mathbf{T}^{\left( 2,1\right) }\equiv (T_{1}^{2},T_{2}^{1})$ and
$\mathbf{T}^{\left( 1,2\right) }\equiv (T_{1},T_{2}^{2})$ are spherically quasinormal. \ Does it follow that $\mathbf{T}$ is spherically quasinormal?
\end{problem}

\begin{problem}
\label{P3}Let $\mathbf{T} \in \mathfrak{H}_{0}$ and assume that $\widehat{\mathbf{T}}=\mathbf{T}=\widetilde{\mathbf{T}}$. \ Does it follow that $\mathbf{T}$ is (jointly) quasinormal?
\end{problem}

We first address Problem \ref{P2} (i). \ The following example
shows that there exists a spherically quasinormal $2$-variable weighted shift $W_{\left( \alpha ,\beta
\right) }$ such that $W_{\left( \alpha ,\beta \right) }^{(2,1)}$ is not
spherically quasinormal. \ Motivated by the ideas in \cite{CuP}, we split the ambient space $\ell ^{2}(%
\mathbb{Z}_{+}^{2})$ into an orthogonal direct sum $\oplus
_{p=0}^{m-1}\oplus _{q=0}^{n-1}\mathcal{H}_{(p,q)}^{(m,n)}$, where%
\begin{equation*}
\mathcal{H}_{(p,q)}^{(m,n)}:=\vee \{e_{(m\ell +p,nk+q)}:k=0,1,2,\cdots ,\ell
=0,1,2,\cdots \}\text{.}
\end{equation*}%
Let $W_{\left( \alpha ,\beta \right) }^{(m,n)}|_{\mathcal{H}%
_{(p,q)}^{(m,n)}} $ be the restriction of $W_{\left( \alpha ,\beta \right)
}^{(m,n)}$ to the space $\mathcal{H}_{(p,q)}^{(m,n)}$. \ Each of $\mathcal{H}%
_{(p,q)}^{(m,n)}$ reduces $T_{1}^{m}$ and $T_{2}^{n}$, and $W_{\left( \alpha
,\beta \right) }^{(m,n)}$ is subnormal if and only if each $W_{\left( \alpha
,\beta \right) }^{(m,n)}|_{\mathcal{H}_{(p,q)}^{(m,n)}}$ is subnormal. \ We
let $\alpha _{(k_{1},k_{2})}^{\left( m,n\right) }|_{\mathcal{H}%
_{(p,q)}^{(m,n)}}$ and $\beta _{(k_{1},k_{2})}^{(m,n)}|_{\mathcal{H}%
_{(p,q)}^{(m,n)}}$ be the weights of $W_{\left( \alpha ,\beta \right)
}^{(m,n)}|_{\mathcal{H}_{(p,q)}^{(m,n)}}$. 

\begin{example} 
\label{ex1}Consider $W_{\left( \alpha ,\beta \right) }$ given by Figure \ref{Figure1}(ii). \ Then $W_{\left( \alpha ,\beta \right) }$ is spherically quasinormal (indeed, a spherical isometry!) but $W_{\left( \alpha ,\beta \right) }^{(2,1)}$ is not spherically quasinormal. \ To establish this, we will apply the results in Theorem \ref{Quasinormal}. \ Suppose that $W_{\left( \alpha ,\beta \right) }^{(2,1)}$ is spherically quasinormal. \ Then $W_{\left( \alpha ,\beta \right) }^{(2,1)}|_{\mathcal{H}_{(0,0)}^{(2,1)}}$
and $W_{\left( \alpha ,\beta \right) }^{(2,1)}|_{\mathcal{H}_{(1,0)}^{(2,1)}} $ are both spherically quasinormal, so that by Theorem \ref{Quasinormal} we have that $\alpha _{(k_{1},k_{2})}^{\left( 2,1\right) }|_{%
\mathcal{H}_{(p,0)}^{(2,1)}}+\beta _{(k_{1},k_{2})}^{(2,1)}|_{\mathcal{H}%
_{(p,0)}^{(2,1)}}$ is constant for all $(k_{1},k_{2})\in \mathbb{Z}_{+}^{2}$
and $p\in \left\{ 0,1\right\} $. \ Consider $W_{\left( \alpha ,\beta \right)
}^{(2,1)}|_{\mathcal{H}_{(0,0)}^{(2,1)}}$ and $\alpha _{(k_{1},0)}^{\left(
2,1\right) }|_{\mathcal{H}_{(0,0)}^{(2,1)}}+\beta _{(k_{1},0)}^{(2,1)}|_{%
\mathcal{H}_{(0,0)}^{(2,1)}}$ for $k_{1}\in \left\{ 0,1\right\} $. \ Observe
that%
\begin{equation*}
\begin{tabular}{l}
$\alpha _{(0,0)}^{\left( 2,1\right) }|_{\mathcal{H}_{(0,0)}^{(2,1)}}+\beta
_{(0,0)}^{(2,1)}|_{\mathcal{H}_{(0,0)}^{(2,1)}}=\frac{8}{9}\neq \frac{44}{45}%
=\alpha _{(1,0)}^{\left( 2,1\right) }|_{\mathcal{H}_{(0,0)}^{(2,1)}}+\beta
_{(1,0)}^{(2,1)}|_{\mathcal{H}_{(0,0)}^{(2,1)}}.$%
\end{tabular}%
\end{equation*}%
By Theorem \ref{Quasinormal}, $W_{\left( \alpha ,\beta
\right) }^{(2,1)}|_{\mathcal{H}_{(0,0)}^{(2,1)}}$ is not spherically
quasinormal, and as a result $W_{\left( \alpha ,\beta \right) }^{(2,1)}$ is not
spherically quasinormal, as desired. \qed
\end{example}

To proceed, we recall some basic results from the theory of truncated moment problems. \ (For more on truncated moment problems we refer to \cite{tcmp1} and \cite{tcmp4}.) \ Given real numbers
\begin{equation*}
\gamma \equiv \gamma ^{(2n)}:=\gamma _{00},\gamma _{01},\gamma _{10},\gamma
_{02},\gamma _{11},\gamma _{20},\cdots ,\gamma _{02n},\cdots ,\gamma _{2n0}
\end{equation*}%
with $\gamma _{00}>0$, the \textit{truncated real moment problem} for $%
\gamma $ entails finding conditions for the existence of a positive Borel
measure $\mu $, supported in $\mathbb{R}_{+}^{2}$, such that
\begin{equation*}
\gamma _{ij}=\int y^{i}x^{j}d\mu(x,y) ,\quad 0\leq i+j\leq n.
\end{equation*}%
Given $\gamma \equiv \gamma ^{(2n)}$, we can build an associated \textit{%
moment matrix} $M(n)\equiv M(n)(\gamma ):=(M[i,j](\gamma ))_{i,j=0}^{n}$,
where
\begin{equation}
M[i,j](\gamma ):=\left(
\begin{array}{llll}
\gamma _{0,i+j} & \gamma _{1,i+j-1} & \cdots & \gamma _{j,i} \\
\gamma _{1,i+j-1} & \gamma _{2,i+j-2} & \cdots & \gamma _{j+1,i-1} \\
\text{ \thinspace \thinspace \quad }\vdots & \text{ \thinspace \thinspace
\quad }\vdots & \ddots & \text{ \thinspace \thinspace \quad }\vdots \\
\gamma _{i,j} & \gamma _{i+1,j-1} & \cdots & \gamma _{i+j,0}%
\end{array}%
\right) .  \label{moment matrix}
\end{equation}%
We denote the successive rows and columns of $M(n)(\gamma )$ by
\begin{equation}
\mathbf{1},X,Y,X^{2},XY,Y^{2},\cdots ,X^{n},\cdots ,Y^{n}.
\label{moment matrix 1}
\end{equation}%
Observe that each block $M[i,j](\gamma )$ is of \textit{Hankel} form, i.e.,
constant in cross-diagonals. 

We now provide a characterization of a class of $2$-variable weighted shifts $W_{\left( \alpha ,\beta \right) } \equiv (T_1,T_2)$ which are spherically quasinormal and have power $(T_1^2,T_2)$ also spherically quasinormal. \ This answers Problem \ref{P2} (ii) in the negative, while identifying the key obstruction, namely the condition $\alpha_{(0,0)}<1$.

\begin{theorem}
\label{Theorem 4}Let $W_{\left( \alpha ,\beta \right) }\equiv (T_1,T_2)$ be spherically quasinormal, and assume that $W_{\left(\alpha ,\beta \right) }^{(2,1)} \equiv (T_1^2,T_2)$ is also spherically quasinormal. \ Assume also that $\alpha_{(3,0)}=1$. \ Then, up to a scalar multiple, either $W_{\left( \alpha ,\beta \right) }$ is the Helton-Howe shift or the Berger measure $\mu$ of $W_{\left( \alpha ,\beta \right) }$ is $2$-atomic of the form
$$
\mu = (1-x_0) \delta_{(0,1+q)} + x_0 \delta_{(1,q)},
$$
where $x_0:=\alpha_{(0,0)}^2$ and $q:=\beta_{(1,0)}^2$. \ In this case, the restriction of $(T_1,T_2)$ to the invariant subspace $\mathcal{N}$ is, up to a constant multiple, the Helton-Howe shift. \ (Here $\mathcal{N}$ denotes the subspace of $\ell^2(\mathbb{Z}_+^2)$ generated by all canonical orthonormal basis vectors $e_{(k_1,k_2)}$ with $k_1\ge1$.)
\end{theorem}

\begin{proof}
Let $x_1:=\alpha _{(1,0)}^2$, $x_2:=\alpha _{(2,0)}^2$, $x_3:=\alpha _{(3,0)}^2$, $
p:=\beta _{(0,0)}^2$, $r:=\beta _{(2,0)}^2$ and $s:=\beta _{(3,0)}^2$. \ By assumption, $x_3=1$. \ Since both $W_{\left( \alpha ,\beta \right) }$ and $W_{\left( \alpha ,\beta \right) }^{(2,1)}$ are spherically quasinormal, Theorem \ref{Quasinormal} readily implies that  
\begin{equation}
\begin{tabular}{l}
$x_0+p=x_1+q=x_2+r=1+s$ and $x_0x_1+p=x_2+r$.
\end{tabular}
\end{equation}
Then
\begin{equation}
\begin{tabular}{l}
$x_0x_1+p=x_0+p \Longrightarrow
x_0\left( x_1-1\right) =0 \Longrightarrow x_1=1$.
\end{tabular}
\end{equation}
Since $W_{\left( \alpha ,\beta \right) }$ is subnormal (and therefore hyponormal), we must have $0\leq x_0\leq x_1\leq x_2 \leq x_3$, and it therefore follows that $x_1=x_2=x_3=1$. \ Since $%
W_{\left( \alpha ,\beta \right) }=(T_{1},T_{2})$ is spherically quasinormal,
it is also subnormal (by Lemma \ref{Lemma 3}), so that $T_{1}$ and $T_{2}$ are
both subnormal operators. \ Since $T_1$ is subnormal, each horizontal row is a subnormal unilateral weighted shift. \ For these shifts, it is well known that the presence of two equal weights readily implies that the shift is of the form $\shift(\alpha_0,\alpha_1,\alpha_1,\alpha_1,\cdots)$, with $\alpha_0 \le \alpha_1$ \cite{Sta2}. \ It follows that
\begin{equation*}
\shift(\alpha _{(0,0)},\alpha _{(1,0)},\cdots )=S_{\alpha _{(0,0)}}\equiv
\shift(\alpha _{(0,0)},1,1,\cdots )
\end{equation*}%
with Berger measures $(1-\alpha _{(0,0)}^{2})\delta _{0}+\alpha _{(0,0)}^{2}\delta
_{1}$. \ Since $W_{\left( \alpha ,\beta \right) }$ is spherically
quasinormal, Theorem \ref{Quasinormal} and the commutativity of $T_{1}$ and $%
T_{2}$ imply that $\shift(\alpha _{(1,k_{2})},\alpha _{(2,k_{2})},\newline \cdots
)=U_{+}$ and $\shift(\beta _{(k_{1},1)},\beta _{(k_{1},2)},\cdots )=\sqrt{q%
}\cdot U_{+}$ for all $k_{1}\geq 1$ and $k_{2}\geq 0$. \ 
This immediately leads to the following column relations in the moment matrix of $W_{(\alpha ,\beta )}$:
\begin{equation}
\left\{
\begin{tabular}{l}
$X^2=X$ \\
$XY=qX$ \\
$X+Y=(1+q) \cdot 1$.
\end{tabular}
\right .
\label{eq202}
\end{equation}
Let us focus now on the $6 \times 6$ moment matrix $M(2)$. \ Since $M(2)$ is positive and recursively generated, we can formally multiply the third column relation in (\ref{eq202}) by $Y$ to obtain
$$
XY+Y^2=(1+q)Y,
$$
and therefore
$$
Y^2=(1+q)Y-qX.
$$
Thus, the columns $Y$, $X^2$, $XY$ and $Y^2$ are all linear combinations of the columns $1$ and $X$. \ It follows that $M(2)$ is a flat extension of $M(1)$, and there exists a unique representing measure $\mu$, with $\card(\supp \mu) = \rank M(1)$, where $\card$ denotes cardinality. (The measure $\mu$ is actually the Berger measure of $( T_1,T_2)$.) \ If $\rank M(1)=1$, then $X=x_0 \cdot 1$ and $Y=p \cdot 1$, so that $x_0=1$ and $p=q$. \ We can then easily show that, up to a constant multiple, $(T_1,T_2)$ is the Helton-Howe shift. \ If instead $\rank M(1)=2$, then $x_0<1$ and therefore $p>q$ (because $x_0+p=1+q$ by the last equation in (\ref{eq202}). \ Also, $\supp \mu$ is the algebraic variety consisting of the intersection of the zero sets of the polynomials associated with the column relations in (\ref{eq202}). The equations to solve (simultaneously) are:
\begin{equation}
\left\{
\begin{tabular}{l}
$x^2=x$ \\
$xy=qx$ \\
$x+y=1+q$.
\end{tabular}
\right .
\label{eq203}
\end{equation}
It follows that $\supp \mu=\{ (0,1+q),(1,q)\}$. \ A simple calculation now reveals that the densities of $\mu$ associated to these two atoms are $1-x_0$ and $x_0$, resp. \ The proof is complete.
\end{proof}

\begin{corollary}
\label{Cor 41}Let $W_{\left( \alpha ,\beta \right) }\equiv (T_1,T_2)$ be spherically quasinormal, and assume that $W_{\left(\alpha ,\beta \right) }^{(2,1)} \equiv (T_1^2,T_2)$ and $W_{\left(\alpha ,\beta \right) }^{(1,2)} \equiv (T_1,T_2^2)$ are also spherically quasinormal. \ Then, up to a scalar multiple, $W_{\left( \alpha ,\beta \right) }$ is the Helton-Howe shift, and therefore $W_{\left( \alpha ,\beta \right) }$ is (jointly) quasinormal.
\end{corollary}

\begin{proof}
We will continue to use the notation from Theorem \ref{Theorem 4} for the weights, and we will also let $b:=\beta_{(0,1)}^2$. \ Assume that $W_{\left( \alpha ,\beta \right) }$ is not the Helton-Howe shift. \ By Theorem \ref{Theorem 4} we know that the Berger measure $\mu$ of $W_{\left( \alpha ,\beta \right) }$ is $2$-atomic with atoms at $(0,1+q)$ and $(1,q)$. \ Since $(T_1,T_2^2)$ is spherically quasinormal, we must have $pb+x_0=q^2+1$. \ Now recall that 
$$
pb=\beta_{(0,1)}^2 \beta_{(0,0)}^2=\gamma_{02}=\int y^2 d\mu = (1-x_0)(1+q)^2+x_0 q^2.
$$
It follows that 
$$
(1-x_0)(1+q)^2+x_0 q^2+x_0=q^2+1,
$$
which is equivalent to 
$$
2q(1-x_0)=0.
$$
Since $q \ne 0$ and $x_0 <1$, we obtain a contradiction. \ The proof is complete.
\end{proof}

We now address Problem \ref{P2}(iii). \ For this, we consider a subnormal $2$-variable weighted shift $(T_1,T_2)$ with $2$-atomic Berger measure written as 
$$
\mu \equiv \sigma \delta_{(s,t)}+\tau \delta_{(u,v)}.
$$
where $s,t,u,v \ge 0$, $s<u$, $t \ne v$ $\sigma, \tau >0$ and $\sigma + \tau =1$. \ We seek concrete necessary and sufficient conditions for the spherical quasinormality of $(T_1,T_2)$, $(T_1^2,T_2)$ and $(T_1,T_2^2)$ in terms of $s,t,u,v,\sigma$ and $\tau$.

\begin{lemma} \label{lem301}
Let $(T_1,T_2)$ and $\mu$ be as above, and recall that $(T_1,T_2)$ is spherically quasinormal if and only if 
\begin{equation}
\label{eq301}
\gamma_{\mathbf{k}+\varepsilon_1}+\gamma_{\mathbf{k}+\varepsilon_2}=C\gamma_{\mathbf{k}}
\end{equation}
for some constant $C>0$ and all $\mathbf{k} \in \mathbb{Z}_+^2$ (Theorem \ref{Quasinormal}). \ Then $(T_1,T_2)$ is spherically quasinormal if and only if 
$$
s+t=u+v.
$$
\end{lemma}

\begin{proof}
In view of (\ref{eq301}), to verify spherical quasinormality we must ensure that 
$$
F(\mathbf{k}):=(\gamma_{\mathbf{k}+\varepsilon_1}+\gamma_{\mathbf{k}+\varepsilon_2})/\gamma_{\mathbf{k}}
$$
is constant for all $\mathbf{k} \in \mathbb{Z}_+^2$. \ This requires two tests: (i) $F(\mathbf{k}+\varepsilon_1)=F(\mathbf{k})$ and (ii) $F(\mathbf{k}+\varepsilon_2)=F(\mathbf{k})$. \ A straightforward calculation using \textit{Mathematica} \cite{Wol} reveals that $F$ satisfies (i) if and only if $F$ satisfies (ii) if and only if  
$$
s^{k_1}t^{k_2}u^{k_1}v^{k_2}(t-v)(s+t-u-v)\sigma \tau=0.
$$
It follows that $(T_1,T_2)$ is spherically quasinormal if and only if 
$$
s+t=u+v,
$$
as desired.
\end{proof}

In a completely similar way, we establish the following result, once again resorting to \textit{Mathematica}.

\begin{lemma} \label{lem302}
Let $(T_1,T_2)$ and $\mu$ be as above. \ Then $(T_1^2,T_2)$ is spherically quasinormal if and only if 
$$
s^2+t=u^2+v.
$$
\end{lemma}

\begin{lemma} \label{lem303}
Let $(T_1,T_2)$ and $\mu$ be as above. \ Then $(T_1^2,T_2)$ is spherically quasinormal if and only if 
$$
s+t^2=u+v^2.
$$
\end{lemma}

\begin{remark}
(i) \ Observe that the conditions in Lemmas \ref{lem301}, \ref{lem302} and \ref{lem303} do not involve the densities $\sigma$ or $\tau$. \ Thus, spherical quasinormality for these $2$-variable weighted shifts depends only on the $\supp \mu$. \newline
(ii) Observe that if the pairs $(s,t)$ and $(u,v)$ satisfy simultaneously the equations in Lemmas \ref{lem301}, \ref{lem302} and \ref{lem303}, then one must have 
$$
s+t=u+v,
$$
$$
s^2+t=u^2+v,
$$
and
$$
s+t^2=u+v^2.
$$
By simple algebraic manipulations (see the Proof of Theorem \ref{thm 3} below), and keeping in mind that $s<u$, one easily obtains that $t=1-s$, $u=1-s$ and $v=s$. \ It follows that any choice of $s$ in the interval $[0,\frac{1}{2})$ and any choice of $\sigma$ produces an example of a $2$-variable weighted shift with $2$-atomic Berger measure and such that $(T_1,T_2)$, $(T_1^2,T_2)$ and $(T_1,T_2^2)$ are each spherically quasinormal. \ Notice that this does not contradict Theorem \ref{Theorem 4}, since here we do not assume that $\alpha_{(3,0)}=1.$
\end{remark}

We are now ready to give an answer to Problem \ref{P2}(iii).

\begin{theorem} \label{thm 3}
There exists a subnormal $2$-variable weighted shift $(T_1,T_2)$ with $2$-atomic Berger measure such that: \newline
(i) $(T_1^2,T_2)$ is spherically quasinormal; \newline
(ii) $(T_1,T_2^2)$ is spherically quasinormal; and \newline
(iii) $(T_1,T_2)$ is not spherically quasinormal.
\end{theorem}

\begin{proof}
By Lemmas \ref{lem301}, \ref{lem302} and \ref{lem303}, it suffices to find a nonnegative real numbers $s,t,u,v$ such that $s<u$, $t \ne v$,
\begin{equation} \label{eq401}
s^2+t=u^2+v,
\end{equation}
\begin{equation} \label{eq402}
s+t^2=u+v^2
\end{equation}
and
\begin{equation} \label{eq403}
s+u \ne t+v.
\end{equation}
It is easy to see that (\ref{eq401}) implies that 
$$
v=s^2+t-u^2.
$$
When this value is inserted into (\ref{eq402}) one gets
$$
t=\frac{1-s^3-s^2u+su^2+u^3}{2(s+u)}.
$$
As a result, (\ref{eq403}) becomes
$$
(u-s)(s+u-1) \ne 0.
$$
Thus, to complete the proof all we need is to choose $s$ and $u$ such that $s+u \ne 1$.
\end{proof}

We next consider Problem \ref{Problem 3}. \ For this, we consider a class $%
\mathcal{A}_{TS}$ of commuting $2$-variable weighted shifts $W_{(\alpha
,\beta )}$ for which the toral and spherical Aluthge transforms agree, that
is, $\widetilde{W}_{(\alpha ,\beta )}=\widehat{W}_{(\alpha ,\beta )}$ (cf.
\cite{CuYo7}, \cite{CuYo6}). \ Using Lemmas \ref{CartAlu} and \ref{PolarAlu}%
, it suffices to restrict attention to the equalities

\begin{equation*}
\sqrt{\alpha _{\mathbf{k}}\alpha _{\mathbf{k}+\mathbf{\varepsilon }_{1}}}%
=\alpha _{\mathbf{k}}\frac{(\alpha _{\mathbf{k}+\mathbf{\epsilon }%
_{1}}^{2}+\beta _{\mathbf{k}+\mathbf{\epsilon }_{1}}^{2})^{1/4}}{(\alpha _{%
\mathbf{k}}^{2}+\beta _{\mathbf{k}}^{2})^{1/4}}\text{ and }\sqrt{\beta _{%
\mathbf{k}}\beta _{\mathbf{k}+\mathbf{\varepsilon }_{2}}}=\beta _{\mathbf{k}}%
\frac{(\alpha _{\mathbf{k}+\mathbf{\epsilon }_{2}}^{2}+\beta _{\mathbf{k}+%
\mathbf{\epsilon }_{2}}^{2})^{1/4}}{(\alpha _{\mathbf{k}}^{2}+\beta _{%
\mathbf{k}}^{2})^{1/4}}
\end{equation*}%
for all $\mathbf{k}\in \mathbb{Z}_{+}^{2}$. \ Thus, we easily see that $%
\widetilde{W}_{(\alpha ,\beta )}=\widehat{W}_{(\alpha ,\beta )}$ if and only
if%
\begin{equation*}
\alpha _{\mathbf{k}+\mathbf{\varepsilon }_{1}}^{2}(\alpha _{\mathbf{k}%
}^{2}+\beta _{\mathbf{k}}^{2})=\alpha _{\mathbf{k}}^{2}(\alpha _{\mathbf{k}+%
\mathbf{\epsilon }_{1}}^{2}+\beta _{\mathbf{k}+\mathbf{\epsilon }_{1}}^{2})%
\text{ and }\beta _{\mathbf{k}+\mathbf{\varepsilon }_{2}}^{2}(\alpha _{%
\mathbf{k}}^{2}+\beta _{\mathbf{k}}^{2})=\beta _{\mathbf{k}}^{2}(\alpha _{%
\mathbf{k}+\mathbf{\epsilon }_{2}}^{2}+\beta _{\mathbf{k}+\mathbf{\epsilon }%
_{2}}^{2})
\end{equation*}%
for all $\mathbf{k}\in \mathbb{Z}_{+}^{2}$, which is equivalent to
\begin{equation*}
\alpha _{\mathbf{k}+\mathbf{\varepsilon }_{1}}\beta _{\mathbf{k}}=\alpha _{%
\mathbf{k}}\beta _{\mathbf{k}+\mathbf{\varepsilon }_{1}}\text{ and }\beta _{%
\mathbf{k}+\mathbf{\varepsilon }_{2}}\alpha _{\mathbf{k}}=\beta _{\mathbf{k}%
}\alpha _{\mathbf{k}+\mathbf{\varepsilon }_{2}}
\end{equation*}%
for all $\mathbf{k}\in \mathbb{Z}_{+}^{2}$. \ If we now recall condition (%
\ref{commuting}) for the commutativity of $W_{(\alpha ,\beta )}$, that is, $%
\alpha _{\mathbf{k}}\beta _{\mathbf{k}+\mathbf{\epsilon _{1}}}=\beta _{%
\mathbf{k}}\alpha _{\mathbf{k}+\mathbf{\epsilon _{2}}}$ for all $\mathbf{k}%
\in \mathbb{Z}_{+}^{2}$, we see at once that $\widetilde{W}_{(\alpha ,\beta
)}=\widehat{W}_{(\alpha ,\beta )}$ if and only if $\alpha _{\mathbf{k}+%
\mathbf{\epsilon _{1}}}=\alpha _{\mathbf{k}+\mathbf{\epsilon _{2}}}$ and $%
\beta _{\mathbf{k}+\mathbf{\epsilon _{2}}}=\beta _{\mathbf{k}+\mathbf{%
\epsilon _{1}}}$ for all $\mathbf{k}\in \mathbb{Z}_{+}^{2}$. \ It follows
that the weight diagram for $W_{(\alpha ,\beta )}$ is completely determined
by the $0$-th row and the weight $\beta _{(0,0)}$ (see Figure \ref{Figure2}(ii)).

\setlength{\unitlength}{1mm} \psset{unit=1mm}
\begin{figure}[th]
\begin{center}
\begin{picture}(135,70)

\psline{->}(20,20)(70,20)
\psline(20,35)(68,35)
\psline(20,50)(68,50)
\psline(20,65)(68,65)
\psline{->}(20,20)(20,70)
\psline(35,20)(35,68)
\psline(50,20)(50,68)
\psline(65,20)(65,68)

\put(12,16){\footnotesize{$(0,0)$}}
\put(31.5,16){\footnotesize{$(1,0)$}}
\put(46.5,16){\footnotesize{$(2,0)$}}
\put(61.5,16){\footnotesize{$(3,0)$}}

\put(25,21){\footnotesize{$\sqrt{x}$}}
\put(40,21){\footnotesize{$1$}}
\put(55,21){\footnotesize{$1$}}
\put(66,21){\footnotesize{$\cdots$}}

\put(25,36){\footnotesize{$1$}}
\put(40,36){\footnotesize{$1$}}
\put(55,36){\footnotesize{$1$}}
\put(66,36){\footnotesize{$\cdots$}}

\put(25,51){\footnotesize{$1$}}
\put(40,51){\footnotesize{$1$}}
\put(55,51){\footnotesize{$1$}}
\put(66,51){\footnotesize{$\cdots$}}

\put(26,66){\footnotesize{$\cdots$}}
\put(41,66){\footnotesize{$\cdots$}}
\put(56,66){\footnotesize{$\cdots$}}
\put(66,66){\footnotesize{$\cdots$}}

\psline{->}(35,14)(50,14)
\put(42,10){$\rm{T}_1$}
\psline{->}(10,35)(10,50)
\put(4,42){$\rm{T}_2$}

\put(11,34){\footnotesize{$(0,1)$}}
\put(11,49){\footnotesize{$(0,2)$}}
\put(11,64){\footnotesize{$(0,3)$}}

\put(20,26){\footnotesize{$\sqrt{xq}$}}
\put(20,41){\footnotesize{$\sqrt{q}$}}
\put(20,56){\footnotesize{$\sqrt{q}$}}
\put(21,66){\footnotesize{$\vdots$}}

\put(35,26){\footnotesize{$\sqrt{q}$}}
\put(35,41){\footnotesize{$\sqrt{q}$}}
\put(35,56){\footnotesize{$\sqrt{q}$}}
\put(36,66){\footnotesize{$\vdots$}}

\put(50,26){\footnotesize{$\sqrt{q}$}}
\put(50,41){\footnotesize{$\sqrt{q}$}}
\put(50,56){\footnotesize{$\sqrt{q}$}}
\put(51,66){\footnotesize{$\vdots$}}

\put(10,6){(i)}


\put(85,8){(ii)}

\psline{->}(95,14)(110,14)
\put(102,10){$\rm{T}_1$}
\psline{->}(77,35)(77,50)
\put(72,42){$\rm{T}_2$}

\psline{->}(80,20)(130,20)
\psline(80,35)(128,35)
\psline(80,50)(128,50)
\psline(80,65)(128,65)

\psline{->}(80,20)(80,70)
\psline(95,20)(95,68)
\psline(110,20)(110,68)
\psline(125,20)(125,68)

\put(75,16){\footnotesize{$(0,0)$}}
\put(91,16){\footnotesize{$(1,0)$}}
\put(106,16){\footnotesize{$(2,0)$}}
\put(121,16){\footnotesize{$(3,0)$}}

\put(85,21){\footnotesize{${\alpha_{00}}$}}
\put(100,21){\footnotesize{${\alpha_{10}}$}}
\put(115,21){\footnotesize{${\alpha_{20}}$}}
\put(126,21){\footnotesize{$\cdots$}}

\put(85,36){\footnotesize{${\alpha_{10}}$}}
\put(100,36){\footnotesize{${\alpha_{20}}$}}
\put(115,36){\footnotesize{${\alpha_{30}}$}}
\put(126,36){\footnotesize{$\cdots$}}

\put(85,51){\footnotesize{${\alpha_{20}}$}}
\put(100,51){\footnotesize{${\alpha_{30}}$}}
\put(115,51){\footnotesize{${\alpha_{40}}$}}
\put(126,51){\footnotesize{$\cdots$}}

\put(85,66){\footnotesize{$\cdots$}}
\put(100,66){\footnotesize{$\cdots$}}
\put(115,66){\footnotesize{$\cdots$}}
\put(126,66){\footnotesize{$\cdots$}}

\put(80,26){\footnotesize{$\beta_{00}$}}
\put(80,41){\footnotesize{$\frac{\alpha_{10}\beta_{00}}{\alpha_{00}}$}}
\put(80,56){\footnotesize{$\frac{\alpha_{20}\beta_{00}}{\alpha_{00}}$}}
\put(81,66){\footnotesize{$\vdots$}}

\put(95,26){\footnotesize{$\frac{\alpha_{10}\beta_{00}}{\alpha_{00}}$}}
\put(95,41){\footnotesize{$\frac{\alpha_{20}\beta_{00}}{\alpha_{00}}$}}
\put(95,56){\footnotesize{$\frac{\alpha_{30}\beta_{00}}{\alpha_{00}}$}}
\put(96,66){\footnotesize{$\vdots$}}

\put(110,26){\footnotesize{$\frac{\alpha_{20}\beta_{00}}{\alpha_{00}}$}}
\put(110,41){\footnotesize{$\frac{\alpha_{30}\beta_{00}}{\alpha_{00}}$}}
\put(110,56){\footnotesize{$\frac{\alpha_{40}\beta_{00}}{\alpha_{00}}$}}
\put(111,66){\footnotesize{$\vdots$}}
\end{picture}
\end{center}
\caption{Weight diagram of the 2-variable weighted shift in Theorem \protect
\ref{Theorem 4} and weight diagram of the 2-variable weighted shift in
Theorem \protect\ref{Theorem 5}, respectively.}
\label{Figure2}
\end{figure}
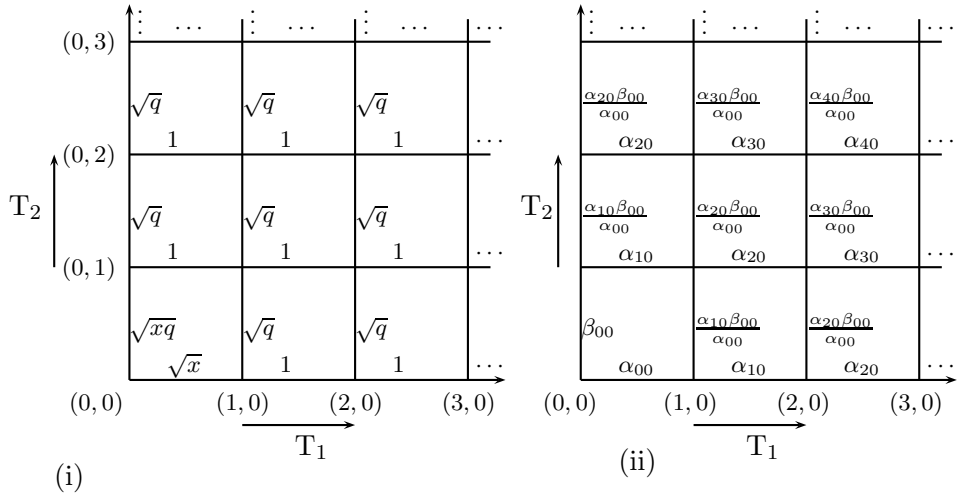

Next, we give an answer to Problem \ref{P3}.

\begin{theorem}
\label{Theorem 5}If $\widetilde{W}_{(\alpha ,\beta )}=\widehat{W}_{(\alpha
,\beta )}=W_{(\alpha ,\beta )}$, then $W_{(\alpha ,\beta )}$ is (jointly)
quasinormal.
\end{theorem}

\begin{proof}
Since $\widetilde{W}_{(\alpha ,\beta )}=\widehat{W}_{(\alpha ,\beta )}$, the
weight diagram of $W_{(\alpha ,\beta )}$ must be of the form shown in Figure \ref{Figure2}(ii). \ Since $\widetilde{W}_{(\alpha ,\beta )}=W_{(\alpha ,\beta )}$, we note
that for $k_{1}\geq 0$
\begin{equation}
\sqrt{\alpha _{(k_{1},0)}\alpha _{(k_{1}+1,0)}}=\alpha
_{(k_{1},0)}\Longrightarrow \alpha _{(k_{1},0)}=\alpha _{(k_{1}+1,0)}\text{.}
\label{equ14}
\end{equation}%
By Theorem \ref{Quasinormal}, we also have that for $k_{1}\geq 0$
\begin{equation}
\beta _{(k_{1},0)}=\beta _{(k_{1}+1,0)}\text{.}  \label{equ15}
\end{equation}%
By the condition (\ref{commuting}) for the commutativity of $W_{(\alpha
,\beta )}$, we have that for $k_{1}\geq 0$
\begin{equation}
\alpha _{(k_{1},1)}=\alpha _{(k_{1}+1,1)}\text{.}  \label{equ16}
\end{equation}%
By Theorem \ref{Quasinormal} again, we have that for $k_{1}\geq 0$
\begin{equation}
\beta _{(k_{1},1)}=\beta _{(k_{1}+1,1)}\text{.}  \label{equ17}
\end{equation}%
Since $W_{(\alpha ,\beta )}\equiv (T_{1},T_{2})$ is spherically quasinormal,
it is subnormal, so that $T_{1}$ and $T_{2}$ are both subnormal. \ By (\ref{equ15}) we have
$\beta_{(0,0)}=\beta_{(1,0)}$. \ Also, from Figure \ref{Figure2}(ii), we have $\beta_{1,0)}=\beta_{(0,1)}$, and
as a result, $\beta_{(0,0)}=\beta_{(0,1)}$. \ Since $T_2$ is subnormal, it follows that $\beta_{(0,0)}=\beta_{(0,k_2)}$ for all $k_2 \ge 1$. \ It is now straightforward that $\beta _{(0,0)}=\beta _{(k_{1},k_{2})}$
for all $k_{1},k_{2}\geq 0$. \ Similarly, $\alpha _{(0,0)}=\alpha _{(k_{1},k_{2})}$ for all $k_1,k_2 \ge 0$. \ It follows that 
\begin{equation*}
W_{(\alpha ,\beta )}\cong \left( I\otimes \alpha _{(0,0)}\cdot U_{+},\beta
_{(0,0)}\cdot U_{+}\otimes I\right) \text{.}
\end{equation*}%
Therefore, by Theorem %
\ref{Theorem 2}, $W_{(\alpha ,\beta )}$ is (jointly) quasinormal, as desired.
\end{proof}




\section{Appendix}

For the reader's convenience, in this section, we gather several well known
auxiliary results which are needed for the proofs of the main results in
this article.

\begin{lemma}
\label{CartAlu} Let $W_{(\alpha ,\beta )}\equiv \left( T_{1},T_{2}\right) $
be a $2$-variable weighted shift. \ Then for all $\mathbf{k}\in \mathbb{Z}%
_{+}^{2}$,
\begin{equation*}
\widetilde{T}_{1}e_{\mathbf{k}}=\sqrt{\alpha _{\mathbf{k}}\alpha _{\mathbf{k}%
+\mathbf{\varepsilon }_{1}}}e_{\mathbf{k}+\mathbf{\varepsilon }_{1}}\text{
and }\widetilde{T}_{2}e_{\mathbf{k}}=\sqrt{\beta _{\mathbf{k}}\beta _{%
\mathbf{k}+\mathbf{\varepsilon }_{2}}}e_{\mathbf{k}+\mathbf{\varepsilon }%
_{2}}
\end{equation*}
\end{lemma}

\begin{lemma}
\label{Lemma 1}\cite{CuYo7} Let $W_{(\alpha ,\beta )}$ be a commuting $2$%
-variable weighted shift. \ Then
\begin{eqnarray}
\widetilde{W}_{(\alpha ,\beta )} &\equiv &\left( \widetilde{T}_{1},%
\widetilde{T}_{2}\right) \text{ is commuting}  \notag \\
&\Longleftrightarrow &\alpha _{\mathbf{k+}\varepsilon _{2}}\alpha _{\mathbf{k%
}+\varepsilon _{1}+\varepsilon _{2}}=\alpha _{\mathbf{k+}\varepsilon
_{1}}\alpha _{\mathbf{k}+2\varepsilon _{2}}  \label{prop1eq}
\end{eqnarray}%
for all $\mathbf{k}\in \mathbb{Z}_{+}^{2}$.
\end{lemma}

\begin{lemma}
\label{Lemma 2}\cite{CuYo7} Consider a $2$-variable weighted shift $%
W_{(\alpha ,\beta )}\equiv \left( T_{1},T_{2}\right) $, and assume that $%
W_{(\alpha ,\beta )}$ is a commuting pair of hyponormal operators. \ Then so
is $\widehat{W}_{(\alpha ,\beta )}$.
\end{lemma}

\begin{lemma}
\label{PolarAlu} Let $W_{(\alpha ,\beta )} \equiv \left(T_{1},T_{2}\right)$
be a $2$-variable weighted shift. \ Then
\begin{equation*}
\widehat{T}_{1}e_{\mathbf{k}}=\alpha_{\mathbf{k}} \frac{(\alpha_{\mathbf{k}+%
\mathbf{\epsilon}_1}^2+\beta_{\mathbf{k}+\mathbf{\epsilon}_1}^2)^{1/4}}{%
(\alpha_{\mathbf{k}}^2+\beta_{\mathbf{k}}^2)^{1/4}} e_{\mathbf{k}+\mathbf{%
\epsilon}_1}
\end{equation*}
and
\begin{equation*}
\widehat{T}_{2}e_{\mathbf{k}}=\beta_{\mathbf{k}} \frac{(\alpha_{\mathbf{k}+%
\mathbf{\epsilon}_2}^2+\beta_{\mathbf{k}+\mathbf{\epsilon}_2}^2)^{1/4}}{%
(\alpha_{\mathbf{k}}^2+\beta_{\mathbf{k}}^2)^{1/4}} e_{\mathbf{k}+\mathbf{%
\epsilon}_2}
\end{equation*}
for all $\mathbf{k} \in \mathbb{Z}_+^2$.
\end{lemma}

\begin{proof}
Straightforward from (\ref{Def-Alu1}).
\end{proof}

\begin{lemma}
\label{Lemma 3} (\cite{AtPo}, \cite{CuYo7}) \ Any spherically quasinormal is
subnormal.
\end{lemma}

\begin{lemma}
\label{Lemma 4} (\cite{CuYo2}, \cite{Yo1}) \ Let $\mu $ be the Berger
measure of a subnormal $2$-variable weighted shift $W_{(\alpha ,\beta
)}\equiv \left( T_{1},T_{2}\right) $, and for $k_{2}\geq 0$ let $\xi
_{k_{2}} $ (resp. $\eta _{k_{1}}$) be the Berger measure of the associated $%
k_{2}$-th horizontal $1$-variable weighted shift $W_{\alpha ^{(k_{2})}}$
(resp. $W_{\beta ^{(k_{1})}}$). \ For every $k_{1},k_{2}\geq 0$ we have
\begin{equation}
\xi _{k_{2}+1}\ll \xi _{k_{2}}\text{ and }\eta _{k_{1}+1}\ll \eta _{k_{1}}%
\text{.}  \label{abs con}
\end{equation}
\end{lemma}

\begin{lemma}
\label{Lemma 5} (\cite{Bro}, \cite{Con}) \ An operator $T\in \mathcal{B}(\mathcal{H})$
with canonical polar decomposition $T=U|T|$ is a (pure) quasinormal operator if
and only if there exists a positive operator $A\in \mathcal{B}(\mathcal{R})$
with $\ker A=\left\{ 0\right\} $ such that $T\cong U_{+}\otimes A$ acting on $\ell^2(\mathbb{Z}_+)
\otimes \mathcal{R}$; thus, $U \cong U_{+} \otimes I$ and $|T| \cong I \otimes A$. \ Furthermore, up to a unitary equivalence, $A$ is uniquely determined.
\end{lemma}

\bigskip
{\bf  Acknowledgments}. \ The authors are deeply grateful to the referee for a careful reading of the paper and for several suggestions and edits that helped to improve the presentation. \ Some of the calculations in this paper were made with the software tool {\it Mathematica} \cite{Wol}.




\end{document}